\documentclass[12pt, twoside]{article}
\usepackage{amsmath,amsthm,amssymb}
\usepackage{times}
\usepackage{enumerate}
\usepackage{bm}
\usepackage[all]{xy}
\usepackage{color}

\def\titlerunning#1{\gdef\titrun{#1}}
\makeatletter
\def\author#1{\gdef\autrun{\def\and{\unskip, }#1}\gdef\@author{#1}}
\def\address#1{{\def\and{\\\hspace*{18pt}}\renewcommand{\thefootnote}{}%
\footnote {#1}}%
\markboth{\autrun}{\titrun}}
\makeatother
\def\email#1{e-mail: #1}
\def\subjclass#1{{\renewcommand{\thefootnote}{}%
\footnote{\emph{Mathematics Subject Classification (2010):} #1}}}
\def\keywords#1{\par\medskip
\noindent\textbf{Keywords.} #1}

\newtheorem{thm}{Theorem}[section]
\newtheorem{cor}[thm]{Corollary}
\newtheorem{lem}[thm]{Lemma}
\newtheorem{prop}[thm]{Proposition}
\theoremstyle{definition}

\numberwithin{equation}{section}

\input xy
\xyoption{all}

\def \C {\mathbb{C}}

\def \de {\delta}
\def \De {\Delta}
\def \la {\lambda}
\def \La {\Lambda}
\def\w {\omega}
\def\Om{\Omega}

\def\pa{\partial}
\def\na {\nabla}

\begin{document}

\baselineskip=17pt

\titlerunning{An $L^{2}$-isolation theorem for Yang-Mills fields on K\"{a}hler surfaces}

\title{An $L^{2}$-isolation theorem for Yang-Mills fields on K\"{a}hler surfaces}

\author{Teng Huang}

\date{}

\maketitle

\address{T. Huang: Department of Mathematics, University of Science and Technology of China,
                   Hefei, Anhui 230026, PR China; \email{oula143@mail.ustc.edu.cn}}


\begin{abstract}
We prove an $L^{2}$ energy gap result for Yang-Mills connections on principal $G$-bundles over compact K\"{a}hler surfaces with positive scalar curvature.\ We prove related results for compact simply-connected Calabi-Yau $2$-folds.


\end{abstract}

\section{Introduction}

A connection on a principal bundle is called Yang-Mills connection when it gives a critical point of the Yang-Mills functional,\ that is,\ it satisfies the Yang-Mills equation
$$d^{\ast}_{A}F_{A}=0.$$
From the Bianchi identity $d_{A}F_{A}=0$,\ a Yang-Mills connection is nothing but a connection whose curvature is harmonic with respect to the covariant exterior derivative $d_{A}$.\ In this article the principal bundle $P$ is always smooth.

Over a $4$-dimensional Riemannnian manifold,\ $F_{A}$ is decomposed into its self-dual and anti-self-dual components,
$$ F_{A}=F^{+}_{A}+F^{-}_{A}$$
where $F^{\pm}_{A}$ denotes the projection onto the $\pm1$ eigenspace of the Hodge star operator.\ A connection is called self-dual (respectively anti-self-dual) if $F_{A}=F^{+}_{A}$ (respectively $F_{A}=F^{-}_{A}$).\ A connection is called an instanton if it is either self-dual or anti-self-dual.\ On compact oriented $4$-manifolds,\ an instanton is always an absolute minimizer of the Yang-Mills energy.\ Not all Yang-Mills connections are instantons.\ See \cite{SS,SSU} for example the $SU(2)$ Yang-Mills connection on $S^{4}$ which are neither self-dual nor anti-self-dual.

It is a very interesting problem for us to study when Yang-Mills connection is anti-self-dual (or self-dual).

In \cite{BL},\ Bourguignon-Lawson proved that Yang-Mills connection on $P\rightarrow X$ a principal $SU(N)$-bundle,\ where $X$ is a four-dimensional anti-self-dual compact Riemannian manifold with positive scalar curvature $S$ must be anti-self-dual,\ provided $|F_{A}^{+}|\leq\frac{S}{12}.$\ It is extended to more general cases in \cite{YZ1,YZ2,Zh}.

Further,\ Min-Oo \cite{Mi} proved that the Yang-Mills connection on $P\rightarrow X$  a principal $SU(2)$-bundle must be anti-self-dual,\ whenever
$\int_{S^{4}}|F^{+}_{A}|^{2}\leq C$,\ where $C$ is suitable positive constant.

Now we assume that the base manifold is a K\"{a}hler surface.\ Then the curvature splits into
$$F_{A}=F^{2,0}_{A}+F^{1,1}_{A}+F^{0,2}_{A},$$
where $F^{p,q}_{A}$ is the $(p,q)$-component.\ We have from the Bianchi identity $\pa_{A}F^{2,0}_{A}=0$ and $\bar{\pa}_{A}F^{0,2}_{A}=0$ with respect to the partial covariant derivatives.\ An anti-self-dual connection relates to a semi-stable holomorphic structure together with an Einstein-Hermitian structure on the associated complex vector bundle \cite{Do}.

The self-dual part $F^{+}_{A}$ is given as $F^{+}_{A}=F^{2,0}_{A}+F^{0,2}_{A}+\frac{1}{2}(\La_{\w}F_{A})\otimes\w$ and the anti-self-dual part $F^{-}_{A}$ is a form of type $(1,1)$ which is orthogonal to $\w$ \cite{DK,Ko}.

Denote by $\mathcal{A}_{YM}$ the space of Yang-Mills connections and $\mathcal{A}_{HYM}$ the space of connections whose curvature satisfies $\hat{F}_{A}=\la Id_{E}$,\ here $\la$ is a constant.\ There spaces are gauge invariant with respect to the group $\mathcal{G}$ of gauge transformations.

The following gives an isolation phenomenon relative to $L^{2}$-norm of $\hat{F}_{A}$ .
\begin{thm}
Let $M$ be a compact K\"{a}hler surface with positive scalar curvature and $A$ be a Yang-Mills connection on a bundle $E$ over $M$ with compact,\ semi-simple Lie group $G$.\ There exist a positive constant $\de=\de(M,E,A)$ with the following significance.\ If $\hat{F}_{A}$ satasfies
$$\|\hat{F}_{A}-\la Id_{E}\|_{L^{2}(M)}\leq\de,$$
where $\la=\frac{2(C_{1}(E)\cdot[w])}{rank(E)[w]^{2}}$,\ then
$$F^{0,2}_{A}=0\quad and \quad \hat{F}_{A}=\la Id_{E}$$
\end{thm}
We have then from this theorem an open subset $$W=\{[A]:\|\hat{F}_{A}-\la Id_{E}\|_{L^{2}(X)}\leq\de\}$$ in the orbit space $\mathcal{A}/\mathcal{G}$ of connections with property $\mathcal{A}_{HYM}/\mathcal{G}=W\cap\mathcal{A}_{YM}/\mathcal{G}$.

In the case that the scalar curvature $S=0$,\ we consider irreducible Yang-Mills connections on compact Calabi-Yau $2$-folds.

A connection is irreducible when it admits nontrivial covariantly constant Lie algebra-valued $0$-form.\ Denote by $\hat{\mathcal{A}}_{YM}$  the space of irreducible Yang-Mills connections and $\hat{\mathcal{A}_{0}}$ the space of irreducible connections whose curvature satisfies $\hat{F}_{A}=0$.\ Then the theorem asserts that $\hat{\mathcal{A}}_{YM}\cap\hat{\mathcal{A}}_{0}=\hat{\mathcal{A}}_{ASD}$,\ here $\hat{\mathcal{A}}_{ASD}$ is the space of irreducible anti-self-dual connections.\ There spaces are gauge invariant with respect to the group $\mathcal{G}$ of gauge transformations.\ So the moduli space of irreducible ansi-self-dual connections $\hat{\mathcal{A}}_{ASD}/\mathcal{G}=\hat{\mathcal{A}}_{YM}/\mathcal{G}\cap\hat{\mathcal{A}}_{0}/\mathcal{G}$.

The following gives an isolation phenomenon relative to $L^{2}$-norm of $\hat{F}_{A}$ over compact Calabi-Yau $2$-fold.
\begin{thm}
Let $M$ be a compact simply-connected Calabi-Yau $2$-fold and $A$ be a irreducible Yang-Mills connection on a bundle $E$ over $M$ with compact,\ semi-simple Lie group $G$.\ There exist a positive constant $\de=\de(M,A)$ with the following significance.\ If $\hat{F}_{A}$ satisfies
$$\|\hat{F}_{A}\|_{L^{2}(M)}\leq\de,$$
then $A$ is anti-self-dual connection.
\end{thm}
\section{Yang-Mills connection over K\"{a}hler manifold}

Let $M$ be a K\"{a}hler manifold with K\"{a}hler metric $g$ and $P\rightarrow M$ be a smooth principal bundle over $M$ with a compact semi-simple Lie group $G$.\ For any connection $A$ on $P$ we have the covariant exterior derivatives $d_{A}:\Om^{k}\rightarrow\Om^{k+1}(\mathfrak{g}_{P})$ where $\Om^{k}(\mathfrak{g}_{P})$ denotes the space of Lie-algebra value $k$-forms.\ Like the canonical splitting the exterior derivatives $d=\pa+\bar{\pa}$,\ decomposes over $M$ into $d_{A}=\pa_{A}+\bar{\pa}_{A}$.\ The curvature splits into $F_{A}=F^{2,0}_{A}+F_{A}^{1,1}+F^{0,2}_{A}$,\ where $F^{p,q}_{A}$ is the $(p,q)$-component.\ Then we get from the Bianchi identity
\begin{equation}\nonumber
\begin{split}
&\pa_{A}F^{2,0}_{A}=\bar{\pa}_{A}F^{0,2}_{A}=0\\
&\bar{\pa}_{A}F^{2,0}_{A}+\pa_{A}F^{1,1}_{A}=\pa_{A}F^{0,2}_{A}+\bar{\pa}_{A}F^{1,1}_{A}=0.\\
\end{split}
\end{equation}
Decompose the curvature,\ $F_{A}$,\ as
$$F_{A}=F^{2,0}_{A}+F^{1,1}_{A0}+\frac{1}{n}\hat{F}_{A}\otimes\w+F^{0,2}_{A},$$
where $\hat{F}_{A}:=\La_{\w}F_{A}$ and $F^{1,1}_{A0}=F^{1,1}_{A}-\frac{1}{n}\hat{F}_{A}\otimes\w$.
\begin{prop}\label{1}
Let $A$ be a Yang-Mills connection on a bundle $E$ over a K\"{a}hler $n$-fold,\ then
\begin{equation}\label{30}
(1)\ 2\bar{\pa}^{\ast}_{A}F^{0,2}_{A}=\sqrt{-1}\bar{\pa}_{A}\hat{F}_{A},
\end{equation}
\begin{equation}
(2)\ 2\pa_{A}^{\ast}F^{2,0}_{A}=-\sqrt{-1}\pa_{A}\hat{F}_{A}.
\end{equation}
\end{prop}
\begin{proof}
We only prove the first identity,\ the second's proof is similar.\ Recall that the Bianchi equations $d_{A}F_{A}=0$,\ we take $(1,2)$ part,\ it implies that
\begin{equation}\nonumber
0=\bar{\pa}_{A}F^{1,1}_{A0}+\frac{1}{n}\bar{\pa}_{A}(\hat{F}_{A}\otimes\w)+\pa_{A}F^{0,2}_{A},
\end{equation}
hence,
\begin{equation}\label{2}
0=\sqrt{-1}\La_{\w}(\bar{\pa}_{A}F^{1,1}_{A0})+\frac{1}{n}\sqrt{-1}\La_{\w}\bar{\pa}_{A}(\hat{F}_{A}\otimes\w)
+\sqrt{-1}\La_{w}\pa_{A}F^{0,2}_{A}.
\end{equation}
The Yang-Mills connection $d^{\ast}_{A}F_{A}=0$,\ we take $(0,1)$ part,\ it implies that
\begin{equation}\label{3}
\pa^{\ast}_{A}(F^{1,1}_{A0}+\frac{1}{n}\hat{F}_{A}\otimes\w)+\bar{\pa}^{\ast}_{A}F^{0,2}_{A}=0.
\end{equation}
By Hodge identities (see \cite{St} 6.12)
\begin{equation}\nonumber
[\La_{\w},\bar{\pa}_{A}]=-\sqrt{-1}\pa^{\ast}_{A}\quad and\quad [\La_{w},\pa_{A}]=\sqrt{-1}\bar{\pa}^{\ast}_{A},
\end{equation}
we can write (\ref{3}) to
\begin{equation}\label{4}
\begin{split}
0&=\sqrt{-1}[\La_{\w},\bar{\pa}_{A}](F^{1,1}_{A0}+\frac{1}{n}\hat{F}_{A}\otimes\w)-\sqrt{-1}[\La_{w},\pa_{A}]F^{0,2}_{A}\\
&=\sqrt{-1}\La_{\w}(\bar{\pa}_{A}F^{1,1}_{A0})+\frac{1}{n}\sqrt{-1}\La_{\w}\bar{\pa}_{A}(\hat{F}_{A}\otimes\w)
-\sqrt{-1}\bar{\pa}_{A}\hat{F}_{A}-\sqrt{-1}\La_{w}\pa_{A}F^{0,2}_{A}.\\
\end{split}
\end{equation}
By (\ref{2}) and (\ref{4}),\ we obtain
\begin{equation}\nonumber
\begin{split}
0&=-2\sqrt{-1}\La_{\w}\pa_{A}F^{0,2}_{A}-\sqrt{-1}\bar{\pa}_{A}\hat{F}_{A}\\
&=2\bar{\pa}^{\ast}_{A}F^{0,2}_{A}-\sqrt{-1}\bar{\pa}_{A}\hat{F}_{A}.\\
\end{split}
\end{equation}
\end{proof}
A connection is irreducible when it admits no nontrivial covariantly constant Lie algebra-value $0$-form,\ i.e.,
$$\ker\{d_{A}:\Om^{0}(\frak{g}_{P})\rightarrow\Om^{1}(\frak{g}_{P})\}=\{0\}.$$
\begin{lem}\label{23}\label{31}
Let $M$ be a compact K\"{a}hler $n$-fold ($n\geq2$) and $A$ be a irreducible Yang-Mills connection on a bundle $E$ over $M$ with compact,\ semi-simple Lie group $G$.\ There exist constants $C=C(\frak{g},A)$ such that either
$$\hat{F}_{A}=0,$$
or
$$\|\hat{F}_{A}\|_{L^{2}(M)}\leq C\|F^{0,2}_{A}\|^{2}_{L^{4}(M)}.$$
\end{lem}
\begin{proof}
From Proposition \ref{1} (\ref{30}),\ we obtain
$$2\bar{\pa}_{A}\bar{\pa}^{\ast}_{A}F^{0,2}_{A}=\sqrt{-1}\bar{\pa}_{A}\bar{\pa}_{A}\hat{F}_{A}=\sqrt{-1}[F^{0,2}_{A},\hat{F}_{A}].$$
Hence
\begin{equation}\label{5}
\|\bar{\pa}^{\ast}_{A}F^{0,2}_{A}\|^{2}_{L^{2}(M)}=\int_{M}\langle\bar{\pa}_{A}\bar{\pa}^{\ast}_{A}F^{0,2}_{A},F^{0,2}_{A}\rangle
=\int_{M}\langle\frac{\sqrt{-1}}{2}[F^{0,2}_{A},\hat{F}_{A}],F^{0,2}_{A}\rangle
\end{equation}
There is a constant $c>0$ depending on the Lie algebra $\mathfrak{g}$ such that $|[X,\bar{X}]|\leq c|X|^{2}$.\ Then
\begin{equation}\label{6}
\begin{split}
&\quad\int_{M}\langle\frac{\sqrt{-1}}{2}[F^{0,2}_{A},\hat{F}_{A}],F^{0,2}_{A}\rangle\\
&=-\int_{M}tr(\frac{\sqrt{-1}}{2}[F^{0,2}_{A},\hat{F}_{A}]\wedge\ast F^{0,2}_{A})=\int_{M}tr(\frac{\sqrt{-1}}{2}\hat{F}_{A}\wedge[F^{0,2}_{A},\ast F^{0,2}_{A}])\\
&=\int_{M}\langle\frac{\sqrt{-1}}{2}\hat{F}_{A},\sum_{ij}[F^{0,2}_{ij},\overline{F^{0,2}_{ij}}]\rangle\\
&\leq c\int_{M}|\hat{F}_{A}||F^{0,2}_{A}|^{2}\leq c\|\hat{F}_{A}\|_{L^{2}(M)}\|F^{0,2}_{A}\|^{2}_{L^{4}(M)}.\\
\end{split}
\end{equation}

here we used the H\"{o}lder inequality $\int_{M}|\hat{F}_{A}||F^{0,2}_{A}|^{2}\leq\|\hat{F}_{A}\|_{L^{2}(M)}\|F^{0,2}_{A}\|^{2}_{L^{4}(M)}$.

Then from (\ref{5}) and (\ref{6}),\ we get,
\begin{equation}\label{7}
\|\bar{\pa}_{A}\hat{F}_{A}\|^{2}_{L^{2}(M)}=4\|\bar{\pa}_{A}^{\ast}F^{0,2}_{A}\|^{2}_{L^{2}(M)}
\leq c\|\hat{F}_{A}\|_{L^{2}(M)}\|F^{0,2}_{A}\|^{2}_{L^{4}(M)}\\
\end{equation}
In the similar way,\ we get
\begin{equation}\label{8}
\|{\pa}_{A}\hat{F}_{A}\|^{2}_{L^{2}(M)}\leq c\|\hat{F}_{A}\|_{L^{2}(M)}\|F^{2,0}_{A}\|^{2}_{L^{4}(M)}.
\end{equation}
Since $A$ is a irreducible connection,\ there exist a positive constant $\la=\la(A)\geq0$ such that
$$\la\|\hat{F}_{A}\|_{L^{2}(M)}\leq\|d_{A}\hat{F}_{A}\|_{L^{2}(M)}=\|\na_{A}\hat{F}_{A}\|_{L^{2}(M)}.$$
From (\ref{7}) and (\ref{8}),\ we get
\begin{equation}\label{9}\nonumber
\begin{split}
\la^{2}\|\hat{F}_{A}\|^{2}_{L^{2}(M)}&\leq\|d_{A}\hat{F}_{A}\|^{2}_{L^{2}(M)}
=\big{(}\|\bar{\pa}_{A}\hat{F}_{A}\|^{2}_{L^{2}(M)}+\|\pa_{A}\hat{F}_{A}\|^{2}_{L^{2}(M)}\big{)}\\
&\leq c\|\hat{F}_{A}\|_{L^{2}(M)}\|F^{0,2}_{A}\|^{2}_{L^{4}(M)}.\\
\end{split}
\end{equation}
Hence,\ we obtain
\begin{equation}\nonumber
\|\hat{F}_{A}\|^{2}_{L^{2}(M)}\leq c\la^{-2}\|\hat{F}_{A}\|_{L^{2}(M)}\|F^{0,2}_{A}\|^{2}_{L^{4}(M)}
\end{equation}
Hence,\ we can choose $C=C(M,A)=c\la^{-2}$.
\end{proof}
\begin{prop}\label{22}
Let $M$ be a compact K\"{a}hler $n$-fold ($n\geq2$) and $A$ be a irreducible Yang-Mills connection on a bundle $E$ over $M$ with compact,\ semi-simple Lie group $G$.\ There exist a positive constant $\de=\de(M,A)$ with the following significance.\ If $\hat{F}_{A}$ satisfies
$$\|\hat{F}_{A}\|_{L^{n}(M)}\leq\de,$$
then
$$\ker\De_{\bar{\pa}_{A}}|_{\Om^{0}(\mathfrak{g}_{P})}=\{0\}.$$
\end{prop}
\begin{proof}
By \cite{DK} Lemma $6.1.7$,\ for any connection $A$,\ we have
$$2\bar{\pa}^{\ast}_{A}\bar{\pa}_{A}s=d^{\ast}_{A}d_{A}s+\sqrt{-1}[\hat{F}_{A},s],$$
where $s\in\Om^{0}(\mathfrak{g}_{P})$.\ For any section $s\in\ker{\De_{\bar{\pa}_{A}}}$,\ we have
$$d^{\ast}_{A}d_{A}s=-\sqrt{-1}[\hat{F}_{A},s],$$
hence,
\begin{equation}\label{28}
\begin{split}
\|d_{A}s\|^{2}_{L^{2}(M)}&=-\int_{M}\langle\sqrt{-1}[\hat{F}_{A},s],s\rangle\\
&\leq\|\hat{F}_{A}\|_{L^{n}(M)}\|s\|^{2}_{L^{\frac{2n}{n-1}}(M)}\leq C_{S}\|\hat{F}_{A}\|_{L^{n}(M)}\|s\|^{2}_{L^{2}_{1}(M)}\\
\end{split}
\end{equation}
The last inequality,\ we used the Sobolev embedding $L^{2}_{1}(M)\hookrightarrow L^{\frac{2n}{n-1}}(M)$ with embedding constant $C_{S}$,\ here $\dim(M)=2n$ .

Since $A$ is a irreducible connection,\ then there exist a positive constant $\la=\la(A)$ such that
\begin{equation}\label{29}
\la\|s\|_{L^{2}(M)}\leq\|d_{A}s\|_{L^{2}(M)}=\|\na_{A}s\|_{L^{2}(M)}.
\end{equation}
By Kato inequality,\ $\big{|}\na|s|\big{|}\leq|\na_{A}s|$ and (\ref{28}),\ (\ref{29}),
hence
\begin{equation}\nonumber
\|d_{A}s\|^{2}_{L^{2}(M)}\leq C_{S}\|\hat{F}_{A}\|_{L^{n}(M)}(1+\la^{-2})\|d_{A}s\|^{2}_{L^{2}(M)}
\end{equation}
We can choose $\de=\frac{1}{2C_{S}(1+\la^{-2})}$,\ then $d_{A}s\equiv0$.\ Since $A$ is a irreducible connection,\ we obtain $s\equiv0$.
\end{proof}

\section{Yang-Mills connection over a K\"{a}hler surface}

We set $\De_{\bar{\pa}_{A}}=\bar{\pa}_{A}\bar{\pa}^{\ast}_{A}+\bar{\pa}_{A}^{\ast}\bar{\pa}_{A}$,\ $\De_{\pa_{A}}={\pa}_{A}{\pa}^{\ast}_{A}+{\pa}_{A}^{\ast}{\pa}_{A}$,\ and $\De_{A}=d^{\ast}_{A}d_{A}+d_{A}d^{\ast}_{A}.$
\begin{prop}(Weitzenb\"{o}ck formula)
Let $M$ be a complete K\"{a}hler surface with Riemannian metric $g$ and $A$ be a connection on a bundle $E$ over $M$.\ Then for each $\phi\in\Om^{0,2}(\mathfrak{g}^{\C}_{P})$,
\begin{equation}\label{12}
\De_{\bar{\pa}_{A}}\phi=\na^{\ast}_{A}\na_{A}\phi+\sqrt{-1}[\hat{F}_{A},\phi]+2S\phi
\end{equation}
where $S$ is the scalar curvature of the metric $g$.
\end{prop}
\begin{proof}
The component of $\bar{\pa}_{A}\phi$ is given by $\na_{A,\bar{\mu}}\phi_{\bar{\nu}\bar{\la}}+\na_{A,\bar{\nu}}\phi_{\bar{\la}\bar{\mu}}{\color{red}+}\na_{A,\bar{\la}}\phi_{\bar{\mu}\bar{\nu}}$,\ where $\na_{A,\bar{\mu}}\cdot=\na_{\bar{\mu}}\cdot+[A_{\bar{\mu}},\cdot]$ ($\na_{A,\mu}\cdot=\na_{\mu}\cdot+[A_{\mu},\cdot]$).\ Then $\bar{\pa}^{\ast}_{A}\bar{\pa}_{A}\phi$ reduces to
$$-\sum g^{\sigma\bar{\tau}}\na_{A,\sigma}
(\na_{A,\bar{\tau}}\phi_{\bar{\mu}\bar{\nu}}+\na_{A,\bar{\mu}}\phi_{\bar{\nu}\bar{\tau}}+\na_{A,\bar{\nu}}\phi_{\bar{\tau}\bar{\mu}}).$$
We have similarly
$$\bar{\pa}_{A}\bar{\pa}_{A}^{\ast}\phi
=-\sum g^{\sigma\bar{\tau}}(\na_{A,\bar{\mu}}\na_{A,\sigma}\phi_{\bar{\tau}\bar{\nu}}-\na_{A,\bar{\nu}}\na_{A,\sigma}\phi_{\bar{\tau}\bar{\mu}}).$$
Then
\begin{equation}\nonumber
\begin{split}
(\De_{\bar{\pa}_{A}}\phi)_{\bar{\mu}\bar{\nu}}
&=-\sum g^{\sigma\bar{\tau}}\na_{A,\sigma}\na_{A,\bar{\tau}}\phi_{\bar{\mu}\bar{\nu}}
-\sum g^{\sigma\bar{\tau}}[\na_{A,\bar{\mu}},\na_{A,\sigma}]\phi_{\bar{\tau}\bar{\nu}}\\
&\quad+\sum g^{\sigma\bar{\tau}}[\na_{A,\bar{\nu}},\na_{A,\sigma}]\phi_{\bar{\tau}\bar{\mu}}\\
&=-\sum g^{\sigma\bar{\tau}}\na_{A,\sigma}\na_{A,\bar{\tau}}\phi_{\bar{\mu}\bar{\nu}}
-\sum g^{\sigma\bar{\tau}}[F_{A,\bar{\mu}\sigma},\phi_{\bar{\tau}\bar{\nu}}]\\
&\quad+\sum g^{\sigma\bar{\tau}}[F_{A,\bar{\nu}\sigma},\phi_{\bar{\tau}\bar{\mu}}]
+\sum(R_{\bar{\mu}}^{\bar{\gamma}}\phi_{\bar{\gamma}\bar{\nu}}-R_{\bar{\nu}}^{\bar{\gamma}}\phi_{\bar{\gamma}\bar{\nu}}).\\
\end{split}
\end{equation}
Since the base manifold is K\"{a}hler surface $S=\frac{1}{2}\sum g^{\bar{\sigma}\tau}R_{\tau\bar{\sigma}}$ and
$\hat{F}_{A}=\sqrt{-1}\sum  g^{\tau\bar{\sigma}}F_{A,\bar{\sigma}\tau}$.\ Thus (\ref{12}) is obtained.
\end{proof}
From Proposition \ref{1} equation (\ref{30}),\ we have $$\De_{\bar{\pa}_{A}}F^{0,2}_{A}=-\frac{1}{2}[\sqrt{-1}\hat{F}_{A},F^{0,2}_{A}],$$
then we obtain
\begin{prop}\label{14}
Let $M$ be a complete K\"{a}hler surface with Riemannian metric $g$ and $A$ be a Yang-Mills connection on a bundle $E$ over $M$.\ Then we have
\begin{equation}\label{13}
\na_{A}^{\ast}\na_{A}F^{0,2}_{A}+\frac{3}{2}[\sqrt{-1}\hat{F}_{A},F^{0,2}_{A}]+2SF^{0,2}_{A}=0
\end{equation}
\end{prop}
From the identity,
$$\ast F_{A}=F^{+}_{A}-F^{-}_{A}=F^{0,2}_{A}+\frac{1}{2}\hat{F}_{A}\otimes\w-F^{1,1}_{A_{0}}+F^{0,2}_{A}.$$
We can write Yang-Mills functional as
\begin{equation}\nonumber
\begin{split}
YM(A)&=4\|F^{0,2}_{A}\|^{2}+\|\hat{F}_{A}\|^{2}+(2C_{2}(E)-C_{1}(E)^{2})\\
&=4\|F^{0,2}_{A}\|^{2}+\|\hat{F}_{A}-\la Id_{E}\|^{2}+(2C_{2}(E)-C_{1}(E)^{2})+\frac{2(C_{1}(E)\cdot[\w])^{2}}{rank(E)[\w]^{2}},\\
\end{split}
\end{equation}
where $\la=\frac{2(C_{1}(E)\cdot[\w])}{rank(E)[\w]^{2}}.$
The energy functional $\|\hat{F}_{A}\|^{2}$ plays an important role in the study of Hermitian-Einstein connections (See \cite{Do} and \cite{UY} ).\ Recall that a connection on a holomorphic vector bundle on a K\"{a}hler manifold is called Hermitian-Einstein if $\hat{F}_{A}=\la Id$.
\begin{prop}(\cite{Do} Proposition 3)\label{15}
Let $A$ be an integrable Yang-Mills connection on an Hermitian vector bundle $E$ over a K\"{a}hler $n$-fold $M$.\ Then $A$ is a direct sum of Hermitian-Einstein connections.\ Further more,\ we can denote $A=\oplus_{i=1}^{l}A_{i}$ where $E=\oplus_{i=1}^{l}E_{i}$ is an orthogonal splitting of $E$,\ and where $\hat{F}_{A_{i}}=\la_{i} Id_{E_{i}}$.
\end{prop}
Under condition of Proposition \ref{15},\ we have
\begin{equation}\nonumber
\|\hat{F}_{A}-\la Id_{E}\|_{L^{2}(M)}=\big{(}\sum_{i=1}^{l}rank(E_{i})|\la_{i}-\la|^{2}Vol(M)\big{)}^{\frac{1}{2}}.
\end{equation}
\begin{cor}\label{16}
Let $A$ be an integrable Yang-Mills connection on an Hermitian vector bundle $E$ over a K\"{a}hler $n$-fold $M$.\ Then there exist a constant $\de=\de(M,E,A)$ such that
$$\|\hat{F}_{A}-\la Id_{E}\|_{L^{2}(M)}\geq\de,$$
or
$$\hat{F}_{A}=\la Id_{E}$$
\end{cor}
\begin{proof}
If we suppose the Yang-Mills connection $A$ is not a Hermitian-Yang-Mills connection.\ Then from Proposition \ref{15},\ the orthogonal splitting,\  $\oplus_{i=1}^{l}E_{i}$,\ of $E$ such that there exist $\la_{i}\neq\la$.\ Hence
$$\|\hat{F}_{A}-\la Id_{E}\|_{L^{2}(M)}\geq\big{(}rank(E_{i})|\la_{i}-\la|^{2}vol(M)\big{)}^{\frac{1}{2}}.$$
then we can choose $\de=(rank(E_{i})|\la_{i}-\la|^{2}vol(M))^{\frac{1}{2}}$.
\end{proof}

Next,\ we gives an isolation theorem relative to $L^{2}$-norm of $\hat{F}_{A}$ on a compact K\"{a}hler surface with positive scalar curvature as follow.
\begin{thm}
Let $M$ be a compact K\"{a}hler surface with positive scalar curvature and $A$ be a Yang-Mills connection on a bundle $E$ over $M$ with compact,\ semi-simple Lie group $G$.\ There exist a positive constant $\de=\de(M,E,A)$ with the following significance.\ If $\hat{F}_{A}$ satasfies
$$\|\hat{F}_{A}-\la Id_{E}\|_{L^{2}(M)}\leq\de,$$
where $\la=\frac{2(C_{1}(E)\cdot[w])}{rank(E)[w]^{2}}$,\ then
$$F^{0,2}_{A}=0\quad and \quad \hat{F}_{A}=\la Id_{E}$$
\end{thm}
\begin{proof}
From Proposition \ref{14},\ we have
\begin{equation}\nonumber
\na_{A}^{\ast}\na_{A}F^{0,2}_{A}+\frac{3\sqrt{-1}}{2}[\hat{F}_{A}-\la Id,F^{0,2}_{A}]+2SF^{0,2}_{A}=0
\end{equation}
and hence
\begin{equation}\label{17}
\int_{M}\langle\na_{A}^{\ast}\na_{A}F^{0,2}_{A},F^{0,2}_{A}\rangle+2S\int_{M}\langle F^{0,2}_{A},F^{0,2}_{A}\rangle
=-\int_{M}\langle\frac{3\sqrt{-1}}{2}[\hat{F}_{A}-\la Id,F^{0,2}_{A}],F^{0,2}_{A}\rangle
\end{equation}
By Kato inequality,\ $\big{|}\na|F^{0,2}_{A}|\big{|}\leq |\na_{A}F^{0,2}_{A}|$,\ we estimate left hand of (\ref{17})
\begin{equation}\label{18}
\int_{M}\langle\na_{A}^{\ast}\na_{A}F^{0,2}_{A},F^{0,2}_{A}\rangle+2S\int_{M}\langle F^{0,2}_{A},F^{0,2}_{A}\rangle\geq C\|F^{0,2}_{A}\|^{2}_{L^{2}_{1}(X)},
\end{equation}
where $C=\min\{1,2S\}$.\ Next,\ we estimate right hand of (\ref{17})
\begin{equation}\label{19}
|\int_{M}\langle\frac{3\sqrt{-1}}{2}[\hat{F}_{A}-\la Id,F^{0,2}_{A}],F^{0,2}_{A}\rangle|\leq\|\hat{F}_{A}-\la Id_{E}\|_{L^{2}(X)}\|F^{0,2}_{A}\|_{L^{4}(X)}
\end{equation}
Then from (\ref{17})--(\ref{19}),\ we get
\begin{equation}\nonumber
\begin{split}
C\|F^{0,2}_{A}\|^{2}_{L^{2}_{1}(X)}&\leq\|\hat{F}_{A}-\la Id_{E}\|_{L^{2}(X)}\|F^{0,2}_{A}\|_{L^{4}(X)}\\
&\leq C_{S}\de\|F^{0,2}_{A}\|_{L^{2}_{1}(X)}.
\end{split}
\end{equation}
where $C_{S}$ is Sobolev constant.\ We choose $\de=\frac{C}{2C_{S}}$,\ then $F^{0,2}_{A}\equiv0.$\ From Corollary \ref{16},\ we also obtain $\hat{F}_{A}=\la Id_{E}$.
\end{proof}
From Proposition \ref{14},\ for each $\phi\in\ker{\De_{\bar{\pa}_{A}}}|_{\Om^{0,2}(\mathfrak{g}^{\C}_{P})}$,\ we have
$$\na_{A}^{\ast}\na_{A}\phi+\sqrt{-1}[\hat{F}_{A}-\la Id_{E},\phi]+2S\phi=0.$$
As above,\ we have
\begin{cor}
Let $M$ be a compact K\"{a}hler surface with positive scalar curvature and $A$ be a Yang-Mills connection on a bundle $E$ over $M$ with compact,\ semi-simple Lie group $G$.\  There exist a positive constant $\de=\de(M)$ with the following significance.\ If $\hat{F}_{A}$ satisfies
$$\|\hat{F}_{A}-\la Id_{E}\|_{L^{2}(M)}\leq\de,$$
where $\la=\frac{2(C_{1}(E)\cdot[w])}{rank(E)[w]^{2}}$,\ then
$$\ker\De_{\bar{\pa}_{A}}|_{\Om^{0,2}(\mathfrak{g}^{\C}_{P})}=\{0\}.$$
\end{cor}
In the case that the scalar curvature $S=0$,\ we consider irreducible Yang-Mills connections on compact Calabi-Yau $2$-folds.

Let $M$ be a compact simply-connected Calabi-Yau $2$-fold,\ with K\"{a}hler form $\w$ and nonzero covariant constant $(2,0)$-form $\theta$ (\cite{GHJ} Definition 4.3 and \cite{Hu} Corollary 4.B.23),\ here $\theta\wedge\bar{\theta}=\frac{\w^{2}}{2}$.\ The form $\theta$ give us a Hodge star $$\ast_{\theta}:\Om^{0,2}(\mathfrak{g}_{P})\rightarrow\Om^{0}(\mathfrak{g}_{P})$$
defined by $\ast_{\theta}\cdot=\ast(\cdot\wedge\theta)$

Let $A$ be a connection on a $G$-bundle $E$ over $M$.\ We can define a section $s\in\Om^{0}(\mathfrak{g}^{\C}_{P})$,\ such that
\begin{equation}\label{20}
\ast(s\wedge\theta)=F^{0,2}_{A},
\end{equation}
hence,\ we have
\begin{equation}\label{25}
-\ast(\bar{\pa}_{A}s\wedge\theta)=\bar{\pa}^{\ast}_{A}F^{0,2}_{A}.
\end{equation}
From (\ref{20}) and (\ref{25}),\ in a direct calculate,\ we have
$$|s|=|F^{0,2}_{A}|\quad and\quad |\bar{\pa}_{A}s|=|\bar{\pa}^{\ast}_{A}F^{0,2}_{A}|.$$
Next,\ we also give an isolation theorem relative to $L^{2}$-norm of $\hat{F}_{A}$ on compact Calabi-Yau $2$-folds.
\begin{thm}
Let $M$ be a compact simply-connected Calabi-Yau $2$-fold and $A$ be a irreducible Yang-Mills connection on a bundle $E$ over $M$ with compact,\ semi-simple Lie group $G$.\ There exist a positive constant $\de=\de(M,A)$ with the following significance.\ If $\hat{F}_{A}$ satisfies
$$\|\hat{F}_{A}\|_{L^{2}(M)}\leq\de,$$
then $A$ is anti-self-dual.
\end{thm}
\begin{proof}
By \cite{DK} Lemma $6.1.7$,\ for any connection $A$,\ we have
$$2\bar{\pa}^{\ast}_{A}\bar{\pa}_{A}s=d^{\ast}_{A}d_{A}s+\sqrt{-1}[\hat{F}_{A},s],$$
where $s\in\Om^{0}(\mathfrak{g}^{\C}_{P})$ satisfies (\ref{20}).\ Hence,\ we have
\begin{equation}\label{26}
\|d_{A}s\|^{2}_{L^{2}(M)}=2\|\bar{\pa}_{A}s\|^{2}_{L^{2}(M)}-\int_{M}\langle\sqrt{-1}[\hat{F}_{A},s],s\rangle.
\end{equation}
We estimate the right hand of (\ref{26}),
\begin{equation}\label{27}
\begin{split}
&\quad2\|\bar{\pa}_{A}s\|^{2}_{L^{2}(M)}-\int_{M}\langle\sqrt{-1}[\hat{F}_{A},s],s\rangle\\
&=2\|\bar{\pa}^{\ast}_{A}F^{0,2}_{A}\|^{2}_{L^{2}(M)}-\int_{M}\langle\sqrt{-1}[\hat{F}_{A},s],s\rangle\\
&=2\int_{M}\langle\bar{\pa}_{A}\bar{\pa}^{\ast}_{A}F^{0,2}_{A},F^{0,2}_{A}\rangle-\int_{M}\langle\sqrt{-1}[\hat{F}_{A},s],s\rangle\\
&=\int_{M}\langle\sqrt{-1}\bar{\pa}_{A}\bar{\pa}_{A}\hat{F}_{A},F^{0,2}_{A}\rangle-\int_{M}\langle\sqrt{-1}[\hat{F}_{A},s],s\rangle\\
&\leq 2\|\hat{F}_{A}\|_{L^{2}(M)}\|F^{0,2}_{A}\|^{2}_{L^{4}(M)}+2\|\hat{F}_{A}\|^{2}_{L^{2}(M)}\|s\|^{2}_{L^{4}(M)}\\
\end{split}
\end{equation}
By Kato inequality,\ $\big{|}\na|s|\big{|}\leq|\na_{A}s|=|d_{A}s|$,\ we have
\begin{equation}
\|F^{0,2}_{A}\|^{2}_{L^{4}(M)}=\|s\|^{2}_{L^{4}(M)}\leq C_{S}\|s\|^{2}_{L^{2}_{1}(M)}
\leq C_{S}\big{(}\|s\|^{2}_{L^{2}(M)}+\|d_{A}s\|^{2}_{L^{2}(M)}\big{)}.
\end{equation}
Since $A$ is a irreducible connection,\ then there exist a positive constant $\la=\la(A)$ such that
\begin{equation}\label{32}
\la\|s\|_{L^{2}(M)}\leq\|d_{A}s\|_{L^{2}(M)}.
\end{equation}
From (\ref{27})--(\ref{32}),\ we obtain
$$\|d_{A}s\|^{2}_{L^{2}(M)}\leq4C_{S}\|\hat{F}_{A}\|_{L^{2}(M)}(1+\la^{-2})\|d_{A}s\|^{2}_{L^{2}(M)}.$$
We choose $\de=\frac{1}{8C_{S}(1+\la^{-2})}$,\ then $d_{A}s\equiv0$.\ Since $A$ is irreducible,\ then $s\equiv0$.
\end{proof}
From Proposition \ref{31},\ we have
\begin{cor}
Let $M$ be a compact simply-connected Calabi-Yau $2$-fold and $A$ be a irreducible Yang-Mills connection on a bundle $E$ over $M$ with compact,\ semi-simple Lie group $G$.\ There exist a positive constant $\de=\de(M,A)$ with the following significance.\ If $F^{0,2}_{A}$ satisfies
$$\|F^{0,2}_{A}\|_{L^{4}(M)}\leq\de,$$
then $A$ is anti-self-dual connection.
\end{cor}

\subsection*{Acknowledgment}
I would like to thank my supervisor Professor Sen Hu for suggesting me to consider this problem,\ and for providing numerous ideas during the course of
stimulating exchanges.\ This work is partially supported by Wu Wen-Tsun Key Laboratory of Mathematics of Chinese Academy of Sciences at USTC.

\end{document}